\documentclass[a4paper,12pt, twoside, reqno]{amsart}
\usepackage{amssymb}
\usepackage{amsmath}
\newtheorem{theorem}{Theorem}
\newtheorem{corollary}[theorem]{Corollary}
\newtheorem{lemma}{Lemma}

\newtheorem{definition}{Definition}
\newtheorem{example}{Example}
\newtheorem{remark}{Remark}[section]

\title{Existence and approximation of fixed points of enriched contractions in quasi-Banach spaces}
\author{Vasile Berinde$^{1,2}$}

\address{$1$ Department of Mathematics and Computer Science Technical University of Cluj-Napoca, North University Centre at Baia Mare, Victoriei 76, 430122 Baia Mare, Romania}

\address{$2$ Academy of Romanian Scientists, 3 Ilfov, 050044, Bucharest, Romania} 

\begin{document}
\maketitle \pagestyle{myheadings} \markboth{Vasile Berinde} {Approximating fixed points of enriched contractions...}
\begin{abstract}
We obtain results on the existence and approximation of fixed points of enriched contractions in quasi-Banach spaces and thus extend the  results obtained in the case of contractions defined on Banach spaces [Berinde, V.; P\u acurar, M. Approximating fixed points of enriched contractions in Banach spaces. J. Fixed Point Theory Appl. 22 (2020), no. 2, Paper No. 38, 10 pp.]. We illustrate the obtained theoretical results by providing an example of an enriched contraction in a quasi-Banach space which is not a Banach space to show that our new results are effective generalizations of the previous ones in literature.

\end{abstract}

\section{Introduction}

In the recent years, the study of quasi-Banach spaces has developed as an extremely flourishing area of research, see  Kalton \cite{Kal03}, as it was very natural to investigate which known properties and results in the well established theory of Banach spaces could be transposed or have a non trivial correspondent in the more general context of quasi-Banach spaces.

From the rich literature devoted to such investigations, we mention the following very recent contributions that cover most of the areas where quasi-Banach spaces form an important setting:
\begin{itemize}
\item stability of solutions to ordinary differential equations and partial differential equations with impulses: Zada and Mashal \cite{Zada};
\item harmonic analysis on Euclidean spaces (harmonic analysis in several variables; maximal functions, Littlewood-Paley theory etc.): Torres et al.  \cite{Tor};
\item Bayesian inference and inverse problems for PDEs: Sullivan \cite{Sul};
\item geometry and structure of normed linear spaces (generalized modulus of convexity, modulus of smoothness, and modulus of Zou-Cui etc.): Kwun et al. \cite{Kwun}; Maligranda \cite{Mal}, Wei \cite{Wei};
\item interpolation of bilinear operators: Grafakos and Mastylo \cite{Graf}; 
\item functional analysis in infinite-dimensional not locally convex spaces: Albiac and Ansorena \cite{Alb}, Kalton \cite{Kal}, S\' anchez \cite{San}; 
\item ...
\end{itemize}

If we refer particularly to the study of partial differential equations, there are many function spaces of basic importance that are not Banach spaces but merely quasi-Banach spaces. Amongst these, there are significant portions of the following familiar scales of spaces: Lebesgue spaces, weak Lebesgue spaces, Lorentz spaces, Hardy spaces, weak Hardy spaces, Lorentz-based Hardy spaces, Besov spaces, Triebel-Lizorkin spaces, and weighted versions of these spaces, see Mitrea et al. \cite{Mitrea}.

On the other hand, some important results on the existence and approximation of the solutions of the fixed point problem
$$
x=Tx
$$
as well as their applications have been recently obtained for various classes of enriched contractive mappings in the setting of a Banach space (or Hilbert space). We collect a few of them in order to illustrate the high interest for the study of this class of mappings:
\begin{itemize}
\item enriched strictly pseudocontractive operators: \' Alvarez \cite{Alvarez}; Berinde \cite{Ber18}; 
\item enriched nonexpansive mappings: Deshmukh et al. \cite{Desh}; Shukla and Pant \cite{Shuk-3}; Shukla and Panicker \cite{Shuk},  \cite{Shuk-2}; Suantai et al. \cite{Suan}; 
\item enriched asymptotically nonexpansive mappings in $\rm CAT(0)$ spaces: Abbas et al. \cite{Abbas-2}; Salisu et al. \cite{Salisu};
\item enriched Banach contractions: Berinde and P\u acurar \cite{BerP20}; Mondal et al. \cite{Mondal};
\item enriched Kannan mappings (with applications to split feasibility and variational inequality problems): Berinde and P\u acurar \cite{BerP21};
\item enriched Chatterjea contractions: Berinde and P\u acurar \cite{BerP21b};
\item enriched \' Ciri\' c-Reich-Rus contractions: Berinde and P\u acurar \cite{BerP21a}, \cite{BerP22};
\item enriched $\varphi$-contractions in Banach spaces: Berinde et al. \cite{Ber22};
\item enriched almost contractions: Berinde and P\u acurar \cite{BerP23};
\item enriched multivalued mappings: Abbas et al. \cite{Abbas};
\item enriched nonexpansive mappings in geodesic spaces: Ali and Jubair \cite{Ali-2};
\item best proximity points for cyclic enriched contractions: Chandok \cite{Chandok};
\item enriched contractions  in partially ordered Banach spaces: Faraji and Radenovi\' c \cite{Faraji};
\item enriched nonexpansive semigroups: Kesahorm and Sintunavarat \cite{Kesahorm};
\item enriched Pre\v si\' c operators: P\u acurar \cite{Pacurar};
\item enriched ordered contractions in noncommutative Banach spaces: Rawat et al. \cite{Rawat};
\item enriched Kannan mappings in $\rm CAT(0)$ spaces: Inuwa et al. \cite{Inuwa};
\item enriched contraction mappings in convex metric spaces: Panicker and Shukla \cite{Panicker};
\item enriched contractions in $\rm CAT(0)$ spaces: Panwar et al. \cite{Panwar};
\item $b$-enriched nonexpansive mapping for solving split variational inclusion problem and fixed point problem in Hilbert spaces: Phairatchatniyom et al. \cite{Phairat};
\item $(b,\theta)$-enriched contractions by means of the degree of nondensifiability: Garc\' ia \cite{Garcia};
\item modified Kannan enriched contraction pair: Anjum and Abbas \cite{Anjum};
\item ...
\end{itemize}

Starting from the fact that Schauder's theorem, which is a topological fixed point theorem, holds for any compact convex set in a quasi-Banach space (in fact, in an $F$-space, see Kalton \cite{Kal}), the following question naturally arises: 
\medskip

{\bf Open question:}
is it possible to transpose metric fixed point theorems from Banach spaces to the setting of a quasi-Banach space ? 
\medskip

The aim of this paper is to answer the above question in the affirmative. More specifically, we are interested to extend some of the fixed point results established in the setting of Banach spaces for various classes of enriched contractions mentioned before to quasi-Banach spaces. 

We start by considering enriched Banach contractions in the present paper.

The motivation of our approach is based on the fact that the standard basic results of Banach space theory which depend only on completeness, such as the Uniform Boundedness Principle, Open Mapping Theorem and Closed Graph Theorem, apply to quasi-Banach spaces, while applications of convexity such as the Hahn-Banach theorem are not applicable, see Maligranda \cite{Mal}, Mitrea et al. \cite{Mitrea}.

In this paper, we investigate the way in which we could extend the results regarding the existence and approximation of fixed points of enriched contraction mappings, established in Berinde and P\u acurar \cite{BerP20} for a Banach space, to the more general setting of a quasi-Banach space. 



One important tool in obtaining our fixed point results in the setting of a quasi-Banach space is Lemma \ref{lem1} which is adapted to quasi-Banach spaces from its original version established in Miculescu and Mihail \cite{Mic}, see also Suzuki \cite{Suz17}, in the case of $b$-metric (quasi-metric) spaces.

To our best knowledge, the present approach is brand new and brings relevant novelty to the fixed point theory, as it allows obtaining important fixed point results that complement many similar in $b$-metric (quasi-metric) spaces which cannot be derived in such a setting, see Berinde and P\u acurar \cite{BerP22a} and references therein.

\section{Preliminaries}

In this section, we present some notations, definitions and auxiliary lemmas which are associated to our main results. 

All vector spaces will be real, although most arguments may be modified to the complex case. 
\begin{definition} \label{def2}
A quasi-norm on a real vector space $X$ is a  map $\|\cdot \|: X\rightarrow [0,\infty)$  satisfying the following conditions:
\smallskip

$(QN_0)$ $\|x\|=0$ if and only if $x=0$;

$(QN_1)$ $\|\lambda x\|=|\lambda|\cdot \|x\|$, for all $x\in X$ and $\lambda\in \mathbb{R}$.

$(QN_2)$ $\|x+y\|\leq C\left[\|x\|+\|y\|\right]$, for all $x,y\in X$, where $C\geq 1$ does not depend on $x,y$;
\end{definition}
\smallskip

According to Pietsch \cite{Pietsch}, the concept of quasi-norm has been introduced by Hyers \cite{Hyers} in 1938, under the name "pseudo-norm" (with $(QN_2)$  formulated differently but in an equivalent way) and, in the above form, by Bourgin \cite{Bourg} in 1943, who actually proposed the label "quasi-norm".

 Obviously, in the particular case $C=1$, a quasi-norm $d$ reduces to the usual notion of {\it norm} but, in general, a norm is not a quasi-norm, see examples in this paper. 
 
 A related notion to a quasi-norm is that of $b$-metric or quasi-metric, which appears to have been first introduced by Vulpe et al. \cite{Vulpe}, but was first known mainly due to the papers by Czerwik \cite{Cze93}, \cite{Cze98}, Bakhtin \cite{Bakh} and Berinde \cite{Ber93}, see Berinde and P\u acurar \cite{BerP22a} for a recent and well documented survey. It is given by the next definition.
 
 \begin{definition}\label{def0}
A function $d:X\times X\rightarrow \mathbb{R}$ is called a {\em quasimetric} on $X$ if it satisfies 
\begin{enumerate}
\item ({\it positivity}) $d(x,y)\geq 0$ and $d(x,y)=0$ if and only if $x=y$; 
\item ({\it symmetry}) $d(x,y)=d(y,x)$, for all $x,y\in X$;
\item ({\it quasi-triangle inequality}) there exists a constant $K\geq 1$ such that 
\begin{equation}\label{eq-tr}
d(x,z)\leq K\left[d(x,y)+d(y,z)\right], \textnormal{ for all } x,y,z\in X.
\end{equation}
\end{enumerate}

If $d$ is a quasimetric on $X$ then the pair $(X,d)$ is called a {\it quasimetric  space}. 
\end{definition}
 
 We note that, as in the case of the classical pair norm - metric (distance), the notions of quasi-norm and that of quasimetric are closely interrelated, in the sense that  any quasi-norm $\|\cdot\|$ on a vector space $X$ induces a quasimetric $d$ by the formula
$$
d(x,y)=\|x-y\|, x,y\in X.
$$

According to Pietsch \cite{Pietsch}, the first example of a quasi-Banach space has been given by Tychonoff \cite{Tych} in 1935, e.g., the space $l_{\frac{1}{2}}$, 
$$
l_{\frac{1}{2}}=\left\{x=(x_1,x_2,\dots,x_n,\dots): \sum_{i=1}^{\infty} \sqrt{|x_i|}<\infty\right\}
$$ 
with the quasi-norm $\|\cdot \|$ defined by
\begin{equation}\label{Tihonov}
\|x \|=\left(\sum_{i=1}^{\infty} \sqrt{|x_i|}\right)^2,
\end{equation}
where $x=(x_1,x_2,\dots,x_n,\dots)$.

It is easily to prove that the above quasinorm   $\|\cdot \|$ on $l_{\frac{1}{2}}$ satisfies the quasi-triangle inequality $(QN_2)$ with $C=2$.
A more general example can be given if we consider $p\in (0,1)$ instead of $p=\frac{1}{2}$ in the considerations above, see for example  \cite{Bour}.

The smallest possible constant $C = C_X \geq 1$ in the quasi-triangle inequality $(QN_2)$ is called the {\em quasi-triangle constant} of $(X,\|\cdot\|)$.

A quasi-norm induces a locally bounded topology on $X$ and conversely, any locally bounded topology is given by a quasi-norm, see for example Maligranda \cite{Mal}.

Let $\|\cdot \|$ be a quasi-norm and $p\in (0,1]$. We call $\|\cdot \|$ a {\it $p$-norm} if we also have
\begin{equation}\label{p-norm}
\|x+y\|^p\leq \|x\|^p+\|y\|^p,\forall x,y\in X.
\end{equation}

A quasi-normed space $X$ is said to be {\it $p$-normable}, $0<p\leq 1$, if there exists an equivalent $p$-norm $\|\cdot \|_*$  on $X$ and a constant $C_1$ such that
$$
\|x_1+\dots+x_n \|_*\leq C_1\left(\| x_1\|_*^p+\dots+\| x_n\|_*^p\right)^{\frac{1}{p}},
$$
for all $x_1,\dots,x_n\in X$.

An $1$-normable space is simply called normable.

Obviously, any $p$-normable space is a quasi-normed space. In fact, we can always assume that the quasi-norm is a $p$-norm for some $p>0$. Indeed, by a theorem of Aoki and Rolewicz (see Kalton \cite{Kal84}), we know that any quasi-norm is equivalent to a $p$-norm, where $p$ satisfies $C=2^{\frac{1}{p}-1}$. More exactly, the formula
$$
\|x\|_1=\inf\left(\sum_{k}\|x_k\|^p\right)^{\frac{1}{p}},
$$
where the infimum is taken over all finite sequences $\{x_n\}\subset X$ satisfying $\sum_k x_k=x$, defines a $p$-norm such that for any $x\in X$, one has
$$
\|x\|_1\leq \|x\|\leq 4^{\frac{1}{p}}\|x\|_1.
$$

If a $p$-norm $\|\cdot \|$ is a quasi-norm on $X$ defining a complete metrizable topology, then $X$ is called a {\em quasi-Banach space}.
\medskip

Regarding the completeness of a quasi-normed space, we state the following important result (Theorem 1.1, Maligranda \cite{Mal}) which emphasises the difference between the Cauchy sequences in a Banach space and Cauchy sequences in a quasi-Banach space.

\begin{theorem}
A quasi-normed space $(X,\|\cdot\|)$ with a quasi-triangle constant $C\geq 1$ is complete (quasi-Banach space) if and only if for every series such that
$\sum\limits_{n=1}^{\infty} \|x_n\|<\infty $ we have $\sum\limits_{n=1}^{\infty} x_n\in X$
and
$$
\left\|
\sum_{n=1}^{\infty} x_n\right\|\leq  \sum_{n=1}^{\infty} C^{n+1} \|x_n\|.
$$
\end{theorem}
\begin{proof}
Using the quasi-triangle inequality, for any $n,p\in \mathbb{N}$, we have
$$
\left\|\sum_{k=n}^{n+p} x_k\right\|\leq C\left(\|x_n\|+C \left\|\sum_{k=n+1}^{n+p} x_k\right\|\right)
$$
$$
\leq C\left[\|x_n\|+C \left(\|x_{n+1}\|+\left\|\sum_{k=n+2}^{n+p} x_k\right\|\right)\right]
$$
$$
\leq \dots \leq \sum_{k=n}^{n+p} C^{k-n+1}\|x_k\|.
$$
Assume now that $X$ is a quasi-Banach space and let $$\sum\limits_{n=1}^{\infty} C^n \|x_n\|=S<\infty.$$

For each $n\in \mathbb{N}$, let $\{S_n\}$ be the sequence of partial sums of the series $\sum\limits_{n=1}^{\infty} x_n$, i.e.,
$$
S_n=\sum\limits_{k=1}^{n} x_k.
$$
Then
$$
\|S_{n+p}-S_n\|=\left\|\sum\limits_{k=n+1}^{n+p} x_k\right\|\leq \sum\limits_{k=n+1}^{n+p} C^{k-n}\|x_k\|\leq \sum\limits_{k=n+1}^{n+p} C^{k}\|x_k\|
$$
$$
=\sum\limits_{k=1}^{n+p} C^{k}\|x_k\|-\sum\limits_{k=1}^{n} C^{k}\|x_k\|\rightarrow S-S=0 \textnormal{ as } n,p\rightarrow \infty,
$$
which proves that $\{S_n\}$ is Cauchy and hence it converges in $X$. This means that
$$
\sum\limits_{n=1}^{\infty} x_n=x\in X.
$$
We also have
$$
\left\|\sum\limits_{k=1}^{n} x_k\right\|\leq \sum\limits_{k=1}^{n} C^{k}\|x_k\|\leq \sum\limits_{k=1}^{\infty} C^{k}\|x_k\|,
$$
and therefore 
$$\limsup\limits_{n\rightarrow \infty}\|S_n\|\leq \sum\limits_{k=1}^{\infty} C^{k}\|x_k\|$$
which yields
$$
\|x\|=\left\| \sum\limits_{k=1}^{\infty} x_k\right\|\leq C \limsup_{n\rightarrow \infty}\left(\|x-S_n\|+\|S_n\|\right)\leq \sum_{n=1}^{\infty} C^{n+1} \|x_n\|.
$$
and the first part of the proof is complete.

For the reverse implication, we refer to the complete proof in Maligranda \cite{Mal}.
\end{proof}
The best known examples of quasi-Banach spaces are the spaces $l_p$ and $L_p (0, 1)$, when $0<p<1$, which are $p$-normable. It is well known, see Kalton \cite{Kal03}, that the dual of $l_p$  is  $l_{\infty}$, while the dual of $L_p (0, 1)$ is  $\{0\}$, i.e., $l_p$  has a separating dual, while $L_p (0, 1)$ has a trivial dual. 

For some examples of quasi-Banach spaces which are important  for the study of Ulam's stability of n$^{th}$ order nonlinear impulsive differential equations, we refer to Zada and Mashal \cite{Zada}.

We close this section by stating the following useful lemma which shall be useful in proving our main results.

\begin{lemma}\label{lem1}
If $\{x_n\} $ is a sequence in a quasi-normed space $(X,\|\cdot \|$) which has the property that there exists $\gamma\in[0,1)$ such that 
\begin{equation}\label{micul}
\|x_{n+1}-x_n\|\leq \gamma \|x_n-x_{n-1}\|, \forall n\in \mathbb{N},
\end{equation}
then $\{x_n\} $ is a Cauchy sequence.
\end{lemma}

\begin{proof}
We simply take 
$$
d(x,y)=\|x-y\|, x,y\in X,
$$
and apply Lemma 2.2 in Miculescu and Mihail \cite{Mic} (or Lemma 6 in Suzuki \cite{Suz17}).
\end{proof}

\section{Enriched contractions in quasi-Banach spaces}

\begin{definition}\label{def3}
Let $(X,\|\cdot\|)$ be a linear quasi-normed space. A mapping $T:X\rightarrow X$ is said to be an {\it enriched contraction} if there exist $b\in[0,+\infty)$ and $\theta\in[0,b+1)$ such that
\begin{equation} \label{eq3}
\|b(x-y)+Tx-Ty\|\leq \theta \|x-y\|,\forall x,y \in X.
\end{equation}
To indicate the constants involved in \eqref{eq3} we shall also call  $T$ a  $(b,\theta$)-{\it enriched contraction}. 
\end{definition}

\begin{remark}
1) As any Banach space is a quasi-Banach space (with $C=1$), the enriched contractions $T$ in a Banach space introduced in Berinde and P\u acurar \cite{BerP20} are enriched contractions in the sense of Definition \ref{def3}.  

2) It is worth mentioning that, like in the case of Banach spaces, any $(b,\theta$)-enriched contraction is continuous.

\end{remark}

Considering a self-mapping $T$ on a convex subset $E$ of a linear space $X$, then for any  $\lambda\in(0,1]$, the so-called averaged mapping $T_{\lambda}$ given by
	\begin{equation} \label{def_T.lambda}
	T_\lambda x=(1-\lambda)x+\lambda Tx, \forall  x \in E,
	\end{equation}
	has the property that $	Fix(\,T_\lambda)=Fix\,(T)$, where $Fix\,(T)$ denotes as usually the set of fixed points of $T$, i.e.,
$$
Fix\,(T)=\{x\in E: Tx=x\}.
$$

\begin{theorem}  \label{th1}
Let $(X,\|\cdot\|)$ be a quasi-Banach space and $T:X\rightarrow X$ a $(b,\theta$)-{\it enriched contraction}. Then

$(i)$ $Fix\,(T)=\{p\}$;

$(ii)$ There exists $\lambda\in (0,1]$ such that the iterative method
$\{x_n\}^\infty_{n=0}$, given by
\begin{equation} \label{eq3a}
x_{n+1}=(1-\lambda)x_n+\lambda T x_n,\,n\geq 0,
\end{equation}
converges to p, for any $x_0\in X$;

\end{theorem}

\begin{proof}
We split the proof into two different cases.

{\bf Case 1.} $b>0$. In this case, let us denote $\lambda=\dfrac{1}{b+1}$. Obviously, $0<\lambda<1$ and the enriched contractive condition \eqref{eq3} becomes
$$
\left \|\left(\frac{1}{\lambda}-1\right)(x-y)+Tx-Ty\right\|\leq \theta \|x-y\|,\forall x,y \in X,
$$
which can be written in an equivalent form as
\begin{equation} \label{eq5}
\|T_\lambda x-T_\lambda y\|\leq c \cdot \|x-y\|,\forall x,y \in X,
\end{equation}
where we denoted $c=\lambda \theta$ and $T_\lambda$ is the averaged mapping defined in (\ref{def_T.lambda}).

Since $\theta\in(0,b+1)$, it follows that $c\in(0,1)$ and therefore by \eqref{eq5}  $T_\lambda$ is a usual Banach contraction with the contraction coefficient $c$ in the quasi-Banach space $X$. 

In view of \eqref{def_T.lambda}, the Krasnoselskij iterative process $\{x_n\}^\infty_{n=0}$ defined by  \eqref{eq3a} is exactly the Picard iteration associated to $T_\lambda$, that is,
\begin{equation} \label{eq4b}
x_{n+1}=T_\lambda x_n,\,n\geq 0.
\end{equation}
Take $x=x_n$ and $y=x_{n-1}$ in  \eqref{eq5} to get
\begin{equation} \label{eq6}
\|x_{n+1}-x_{n}\|\leq c \cdot \|x_{n}-x_{n-1}\|,\,n\geq 1.
\end{equation}

Now, by Lemma \ref{lem1} it follows that $\{x_n\}^\infty_{n=0}$ is a Cauchy sequence and hence it is convergent in the quasi-Banach space $(X,\|\cdot\|)$. Let us denote
\begin{equation} \label{eq10a}
p=\lim_{n\rightarrow \infty} x_n.
\end{equation}
By letting  $n\rightarrow \infty$ in \eqref{eq4b} and using the continuity of $T_\lambda$ (which follows by the continuity of $T$), we immediately obtain
$$
p=T_\lambda p,
$$
that is, $p\in Fix\,(T_\lambda)$. 

Next, we prove that $p$ is the unique fixed point of $T_\lambda$. Assume that $q\neq p$ is another fixed point of $T_\lambda$. Then, by \eqref{eq5} 
$$
0<\|p-q\|\leq c \cdot \|p-q\|<\|p-q\|,
$$
a contradiction. Hence $Fix\,(T_\lambda)=\{p\}$ and since, by \eqref{eq4b},  $Fix\,(T)=Fix(\,T_\lambda)$, claim $(i)$ is proven.

Conclusion $(ii)$ follows by \eqref{eq10a}.

{\bf Case 2.} $b=0$. In this case,  $\lambda=1$, $c=\theta$ and we proceed like in Case 1 but with $T(=T_1)$ instead of $T_\lambda$, when the Krasnoselskij iteration \eqref{eq3a} reduces in fact to the simple Picard iteration associated to $T$,
$$
x_{n+1}=T x_n,\,n\geq 0.
$$
\end{proof}

\begin{example}[ ] \label{ex1}
Let $X=\mathbb{R}^2$ and, for $x=(x_1,x_2)\in \mathbb{R}^2$, consider the functionals $\|\cdot\|_{a,p}$
$$
\|x\|_{a,p}=
\begin{cases}
\|x\|_p=\left(|x_1|^p+|x_2|^p\right)^{1/p},\,\textnormal{ if } x_2\neq 0\\
a|x_1|,\,\textnormal{ if } x_2= 0,
\end{cases}
$$
where $1\leq p \leq \infty$ and $a\neq 1$. It is well known, see Maligranda \cite{Mal}, that $\|\cdot\|_{a,p}$ is a quasi-norm on $X$, with the quasi-triangle inequality constant
$$
C=\max\left\{a,\frac{1}{a}\right\}.
$$
Let $E=\left[\frac{1}{2},2\right]^2\subset X$ and $T:E\rightarrow E$ be given by $T(x,y)=(1-x,1-y)$, for all $(x,y)\in E$. 

First, we observe  that $T$ is not a Banach contraction on $(X,\|\cdot\|_{a,p})$, since it is merely nonexpansive (in fact $T$ is an isometry). 

Next, we prove that $T$ is an enriched contraction with respect to the quasi-norm $\|\cdot\|_{a,p}$. Indeed, for $u=(x_1,y_1), v=(x_2,y_2)\in E$ the left hand side of the enriched contraction condition \eqref{eq3} will be
$$
\|Tu-Tv\|=\|T(x_1,y_1)-T(x_2,y_2)\|
$$
$$
=\|\left((b-1)(x_1-x_2),(b-1)(y_1-y_2)\right)\|=|b-1|\cdot \|u-v\|.
$$
Now, in order to have \eqref{eq3} satisfied, we take $\theta=1-b$, with $0<b<1$. 

So, for any $b\in (0,1)$, $T$ is a $(b,1-b)$ enriched contraction on $(X,\|\cdot\|_{a,p})$.


It is easily seen that  $Fix\,(T)=\{(1/2,1/2)\}$.
\end{example}

\begin{remark}\label{rem1}
In the particular case when the quasi-triangle constant $C=1$, by Theorem \ref{th1} one obtains the first part of the main result in Berinde and P\u acurar \cite{BerP20}. We also note that, by following the same technique like the one in proving Theorem 2.4 in Berinde and P\u acurar \cite{BerP20}, we can also obtain the error estimate that corresponds to \eqref{3.2-1a} for the case of quasi-Banach spaces.
\end{remark}

\begin{corollary} [Berinde and P\u acurar \cite{BerP20}, Theorem 2.4]\label{cor1}
Let $(X,\|\cdot\|)$ be a Banach space and $T:X\rightarrow X$ a $(b,\theta$)-{\it enriched contraction}. Then

$(i)$ $Fix\,(T)=\{p\}$;

$(ii)$ There exists $\lambda\in (0,1]$ such that the iterative method
$\{x_n\}^\infty_{n=0}$, given by
\begin{equation} \label{eq3aa}
x_{n+1}=(1-\lambda)x_n+\lambda T x_n,\,n\geq 0,
\end{equation}
converges to p, for any $x_0\in X$;

$(iii)$ The following estimate holds
\begin{equation}  \label{3.2-1a}
\|x_{n+i-1}-p\| \leq\frac{c^i}{1-c}\cdot \|x_n-
x_{n-1}\|\,,\quad n=0,1,2,\dots;\,i=1,2,\dots,
\end{equation}
where $c=\dfrac{\theta}{b+1}$.

\end{corollary}

\section{Local and asymptotic versions of enriched contraction mapping principle} 

The following example shows that  there exist mappings $T$ which are not contractions but a certain iterate of them is  a contraction.

\begin{example} (Examples 1.3.1, \cite{Rus01a}) \label{ex2}
Let $X=\mathbb{R}$ and $T:X\rightarrow X$ be given by $Tx=0$, if $x\in (-\infty,2]$ and $Tx=-\dfrac{1}{3}$, if $x\in (2,+\infty)$. Then $T$ is not a contraction (being discontinuous) but $T^2$ is a contraction.
\end{example}

In such a case, we cannot apply the classical Picard-Banach contraction mapping principle and thus the following fixed point theorem is useful, see for example Theorem 1.3.2 in Rus \cite{Rus01a}.

\begin{theorem}\label{th2}
Let $(X,d)$ be a complete metric space and let $T:X\rightarrow X$ be a mapping. If there exists a positive integer $N$ such that $T^N$ is a contraction, then $Fix\,(T)=\{x^*\}$.
\end{theorem}

Our first aim in this section is to obtain a similar result for the more general case of enriched contractions in the setting of a quasi-Banach space.

\begin{theorem}\label{th3}
Let $(X,\|\cdot\|)$ be a quasi-Banach space and let $U:X\rightarrow X$ be a mapping with the property that there exists a positive integer $N$ such that $U^N$ is a $(b,\theta$)-enriched contraction. Then 

$(i)$ $Fix\,(U)=\{p\}$;

$(ii)$ There exists $\lambda\in (0,1]$ such that the iterative method
$\{x_n\}^\infty_{n=0}$, given by
$$
x_{n+1}=(1-\lambda)x_n+\lambda U^N x_n,\,n\geq 0,
$$
converges to $p$, for any $x_0\in X$.
\end{theorem}

\begin{proof}
We apply Theorem \ref{th1} (i) for the mapping $T=U^N$ and obtain that $Fix\,(U^N)=\{p\}$. We also have
$$
U^N(U(p))=U^{N+1}(p)=U(U^N(p))=U(p),
$$
which shows that $U(p)$ is a fixed point of $U^N$. But $U^N$ has a unique fixed point, $p$, hence $U(p)=p$ and so $p\in Fix\,(U)$. 

The remaining part of the proof follows by Theorem \ref{th1}. 
\end{proof}

One of the most interesting generalizations of the contraction mapping principle is the so-called Maia fixed point theorem, see Maia \cite{Maia}, which was obtained by splitting the assumptions in the contraction mapping principle among two metrics defined on the same set. 

Recall that a map $T:X\rightarrow X$ satisfying

$(p1)$ $T$ has a unique fixed point $p$ in $X$;

 $(p2)$ The Picard iteration $\{x_n\}^\infty_{n=0}$ converges to $p$, for any $x_0\in X$,
 
 is said to be a \emph{Picard operator}, see Rus \cite{Rus01a}, \cite{Rus03}, \cite{Rus14} for examples and more details.

 \begin{theorem} (Rus \cite{Rus01a}, Theorem 1.3.10)\label{th4}
 Let $X$ be a nonempty set, $d$ and $\rho$ two metrics on $X$ and $T : X \rightarrow X$ a mapping. Suppose that
 
 (i) $d(x,y)\leq \rho(x,y) $, for each $x,y \in X$;
 
 (ii) $(X, d)$ is a complete metric space;
 
 (iii) $T : X \rightarrow X$ is continuous  with respect to the metric $d$;
 
 (iv) $T$ is a contraction mapping with respect to the metric $\rho$. 
 
 Then $T$ is a Picard operator.
 \end{theorem}

The next theorem is an extension of Theorem \ref{th4} to the class of enriched contractions but it is established in the particular case of a quasi-Banach space.

\begin{theorem} \label{th5} 
 Let $X$ be a linear space, $\|\cdot \|_d$ and $\|\cdot\|_{\rho}$ two norms on $X$ and $T : X \rightarrow X$ a mapping. Suppose that
 
 (i) $\|x-y\|_d\leq \|x-y\|_{\rho} $, for each $x,y \in X$;
 
 (ii) $(X, \|\cdot \|_d)$ is a quasi-Banach space;
 
 (iii) $T : X \rightarrow X$ is continuous  with respect to the norm $\|\cdot\|_{d}$;
 
 (iv) $T$ is a $(b,\theta)$-enriched contraction mapping with respect to the norm $\|\cdot\|_{\rho}$. 
 
 Then $T$ is a Picard operator.
 \end{theorem}

\begin{proof}
Let $x_0\in X$. By (iv), we deduce similarly to the proof of Theorem \ref{th1} that $\{T_{\lambda}^n x_0\}$ is a Cauchy sequence in $(X,\|\cdot\|_{\rho})$, where as usually $\lambda=\dfrac{1}{b+1}$. By (i), $\{T_{\lambda}^n x_0\}$ is a Cauchy sequence in $(X,\|\cdot\|_{d})$ and by (ii) it converges. Let
$$
x^*=\lim_{n\rightarrow \infty} T_{\lambda}^n x_0.
$$
By (iii) we obtain that $x^*\in Fix\,(T_{\lambda})$ and by (iv) that $Fix\,(T_{\lambda})=\{x^*\}$. Since $Fix\,(T_{\lambda})=Fix\,(T)$, the conclusion follows.
\end{proof}

\section{Conclusions}
\indent

1. We introduced the class of enriched contractions in quasi-Banach spaces and obtained an extension of the main result in Berinde and P\u acurar \cite{BerP20} regarding the existence and approximation of fixed points of enriched contractions in Banach spaces;

2. We have shown that any enriched contraction in a quasi Banach space has a unique fixed point that can be approximated by means of some Kransnoselskij iteration. In particular, by our fixed point results established in this paper we obtain the classical Banach contraction principle in the setting of a Banach space.

3. It is worth mentioning that enriched contractions preserve a fundamental property of Picard-Banach  contractions, namely any enriched contraction has a unique fixed point and is continuous (one can easily show this considering the definition).

4. We also obtained a fixed point theorem of Maia type which extends the Maia fixed point theorem (Theorem \ref{th4}) from the class of enriched contractions in Banach spaces to the class of enriched contractions in quasi-Banach spaces (Theorem \ref{th5}).

5. To our best knowledge, these are the first fixed point results obtained in the setting of a quasi-Banach space. They bring relevant novelty to the fixed point theory, as it allows obtaining important fixed point results that complement many similar in $b$-metric (quasi-metric) spaces which cannot be derived in such a setting, see Berinde and P\u acurar \cite{BerP22a} and references therein.

6. We do hope that our approach will open a new direction of study in the metrical fixed point theory, i.e., to extend known results from Banach spaces to relevant ones in quasi-Banach spaces.



\begin{thebibliography}{99}
\bibitem{Abbas} Abbas, M.; Anjum, R.; Tahir, M. H. Fixed point theorems of enriched multivalued mappings via sequentially equivalent Hausdorff metric. {\em Topol. Algebra Appl.} {\bf 11} (2023), no. 1, Paper No. 20220136, 11 pp.

\bibitem{Abbas-2} Abbas, M.; Anjum, R.; Ismail, Nimra. Approximation of fixed points of enriched asymptotically nonexpansive mappings in $\rm CAT(0)$ spaces. {\em Rend. Circ. Mat. Palermo (2)} {\bf 72} (2023), no. 4, 2409--2427.

\bibitem{Alb} Albiac, F.; Ansorena, J. L. Uniqueness of unconditional basis of infinite direct sums of quasi-Banach spaces. {\em Positivity} {\bf 26} (2022), no. 2, Paper No. 35, 43 pp.

\bibitem{Ali} Ali, J.; Jubair, M. Existence and estimation of the fixed points of enriched Berinde nonexpansive mappings. {\em Miskolc Math. Notes} {\bf 24} (2023), no. 2, 541--552.

\bibitem{Ali-2} Ali, J.; Jubair, M.  Fixed points theorems for enriched non-expansive mappings in geodesic spaces. {\em Filomat} {\bf 37} (2023), no. 11, 3403--3409.

\bibitem{Alvarez} \' Alvarez, E. D. J. On monotone pseudocontractive operators and Krasnoselskij iterations in an ordered Hilbert space. {\em Arab. J. Math. (Springer)} {\bf 12} (2023), no. 2, 297--307.

\bibitem{Anjum} Anjum, R.; Abbas, M. Common fixed point theorem for modified Kannan enriched contraction pair in Banach spaces and its applications. {\em Filomat} {\bf 35} (2021), no. 8, 2485--2495.

\bibitem{Abb21} Abbas, M.; Anjum, R.; Berinde, V. Equivalence of Certain Iteration Processes Obtained by Two New Classes of Operators. Mathematics 9 (2021), no. 18, Article Number 2292.

\bibitem{Abb21a} Abbas, M.; Anjum, R.; Berinde, V. Enriched Multivalued Contractions with Applications to Differential Inclusions and Dynamic Programming. Symmetry-Basel, 13 (2021), no. 8, Article Number 1350.


\bibitem{Bakh} Bakhtin, I. A. Contracting mapping principle in an almost metric space. (Russian) {\it Funkts. Anal.} {\bf 30} (1989), 26--37.

\bibitem{Ber93} Berinde, V. Generalized contractions in quasimetric spaces. {\it Seminar on Fixed Point Theory}, 3--9, Preprint, 93-3, "Babe\c s-Bolyai'' Univ., Cluj-Napoca, 1993.

\bibitem{Ber18} Berinde, V. Weak and strong convergence theorems for the Krasnoselskij iterative algorithm in the class of enriched strictly pseudocontractive operators. {\em An. Univ. Vest Timi\c s. Ser. Mat.-Inform.} {\bf 56} (2018), no. 2, 13--27.



\bibitem{Ber19} Berinde, V. Approximating fixed points of enriched nonexpansive mappings by Krasnoselskij iteration in Hilbert spaces. Carpathian J. Math. 35 (2019), no. 3, 293--304.

\bibitem{Ber20} Berinde, V. Approximating fixed points of enriched nonexpansive mappings in Banach spaces by using a retraction-displacement condition. Carpathian J. Math. 36 (2020), no. 1, 27--34.

\bibitem{Ber22} Berinde, V. Maia type fixed point theorems for some classes of enriched contractive mappings in Banach spaces. Carpathian J. Math. 38 (2022), no. 1, 35--46.

\bibitem{Ber22c} Berinde, V. A Modified Krasnosel'ski\v i-Mann Iterative Algorithm for Approximating Fixed Points of Enriched Nonexpansive Mappings. Symmetry-Basel, 14 (2022), no. 1, Article number 123.

\bibitem{Ber23} Berinde, V. Approximating fixed points of demicontractive mappings via the quasi-nonexpansive case. Carpathian J. Math. 39 (2023), no. 1, 73--85.

\bibitem{Ber24} Berinde, V. On a useful lemma that relates quasi-nonexpansive and demicontractive mappings in Hilbert spaces. {\em Creat. Math. Inform.} {\bf 33} (2024), no. 1, 7--21.


\bibitem{Ber22b} Berinde, V.; Harjani, J.; Sadarangani, K. Existence and Approximation of Fixed Points of Enriched $\varphi$-Contractions in Banach Spaces. Mathematics,  21 (2022), no. 10, Article No 4138.

\bibitem{BerP20} Berinde, V.; P\u acurar, M. Approximating fixed points of enriched contractions in Banach spaces. {\em J. Fixed Point Theory Appl.} {\bf 22} (2020), no. 2, Paper No. 38, 10 pp.

\bibitem{BerP21} Berinde, V.; P\u acurar, M. Kannan's fixed point approximation for solving split feasibility and variational inequality problems. {\em J. Comput. Appl. Math.} {\bf 386} (2021), Paper No. 113217, 9 pp.

\bibitem{BerP21a} Berinde, V.; P\u acurar, M. Fixed point theorems for enriched \' Ciri\' c-Reich-Rus contractions in Banach spaces and convex metric spaces. {\em Carpathian J. Math.} 37 (2021), no. 2, 173--184.

\bibitem{BerP21b} Berinde, V.; P\u acurar, M.  Approximating fixed points of enriched Chatterjea contractions by Krasnoselskij iterative algorithm in Banach spaces. {\em J. Fixed Point Theory Appl.} {\bf 23} (2021), no. 4, Paper No. 66, 16 pp.

\bibitem{BPac21c} Berinde, V.; P\u acurar, M.  Existence and Approximation of Fixed Points of Enriched Contractions and Enriched $\varphi$-Contractions. Symmetry-Basel, 13 (2021), no. 3, Article Number 498.

\bibitem{BPac21d} Berinde, V.; P\u acurar, M.  Fixed points theorems for unsaturated and saturated classes of contractive mappings in Banach spaces. {\em Symmetry-Basel}, {\bf 13} (2021), no. 4, Article Number 713.


\bibitem{BerP22} Berinde, V.; P\u acurar, M. A new class of unsaturated mappings: \' Ciri\' c-Reich-Rus contractions. {\em An. \c Stiin\c t. Univ. "Ovidius'' Constan\c ta Ser. Mat.} {\bf 30} (2022), no. 3, 37--50.

\bibitem{BerP22a} Berinde, V.; P\u acurar, M. The early developments in fixed point theory on b-metric
spaces: a brief survey and some important related aspects. {\em Carpathian J. Math.} {\bf 38} (2022), no. 3, 523--538.

\bibitem{BerP23} Berinde, V.; P\u acurar, M.  Krasnoselskij-type algorithms for fixed point problems and variational inequality problems in Banach spaces. {\em Topology Appl.} {\bf 340} (2023), Paper No. 108708, 15 pp.







\bibitem{BPR23} Berinde, V.; Petru\c sel, A.; Rus, I. A. Remarks on the terminology of the mappings in fixed point iterative methods in metric spaces. {\it Fixed Point Theory} {\bf 24} (2023), no. 2, 525--540.


\bibitem{Bour} Bourbaki, N. {\it Topologie G\' en\' erale}, Herman, Paris, 1974.

\bibitem{Bourg} Bourgin, D. G. Linear topological spaces. {\it Amer. J. Math.} {\bf 65} (1943), 637--659. 

\bibitem{Chandok} Chandok, S. Existence and convergence of best proximity points for cyclic enriched contractions. {\em Linear Nonlinear Anal.} {\bf 8} (2022), no. 2, 185--195.

\bibitem{Cze93} Czerwik, S. Contraction mappings in $b$-metric spaces. {\it Acta Math. Inform. Univ. Ostraviensis} {\bf 1} (1993), 5--11. 

\bibitem{Cze98} Czerwik, S. Nonlinear set-valued contraction mappings in $b$-metric spaces. {\it Atti Sem. Mat. Fis. Univ. Modena} {\bf 46} (1998), no. 2, 263--276.

\bibitem{Desh} Deshmukh, A.; Gopal, D.; Rakocevi\' c, V. Two new iterative schemes to approximate the fixed points for mappings. {\em Int. J. Nonlinear Sci. Numer. Simul.} {\bf 24} (2023), no. 4, 1265--1309.

\bibitem{Faraji} Faraji, H.; Radenovi\' c, S. Some fixed point results of enriched contractions by Krasnoselskij iterative method in partially ordered Banach spaces. {\em Trans. A. Razmadze Math. Inst.} {\bf 177} (2023), no. 1, 19--26.

\bibitem{Garcia} Garc\' ia, G. A generalization of the $(b,\theta)$-enriched contractions based on the degree of nondensifiability. {\em Asian-Eur. J. Math.} {\bf 15} (2022), no. 9, Paper No. 2250168, 12 pp.

\bibitem{Graf} Grafakos, L.; Mastylo, M. Interpolation of bilinear operators between quasi-Banach spaces. {\em Positivity} {\bf 10} (2006), no. 3, 409--429.

\bibitem{Hyers} Hyers, D. H. A note on linear topological spaces. {\it Bull. Amer. Math. Soc.}  {\bf 44} (1938) pp. 76--80.

\bibitem{Inuwa} Inuwa, A. Y.; Kumam, P.; Chaipunya, P.; Salisu, S. Fixed point theorems for enriched Kannan mappings in $\rm CAT(0)$ spaces. {\em Fixed Point Theory Algorithms Sci. Eng.} {\bf 2023}, Paper No. 13, 21 pp.

\bibitem{Kal} Kalton, N. J. The existence of primitives for continuous functions in a quasi-Banach space. {\em Atti Sem. Mat. Fis. Univ. Modena} {\bf 44} (1996), no. 1, 113--117.

\bibitem{Kal03} Kalton, N. J. {\em Quasi-Banach spaces. Handbook of the geometry of Banach spaces}, Vol. 2, 1099--1130, North-Holland, Amsterdam, 2003.

\bibitem{Kal84} Kalton, N. J.; Peck, N. T.; Roberts, James W. {\em An $F$-space sampler}. London Mathematical Society Lecture Note Series, 89. Cambridge University Press, Cambridge, 1984.

\bibitem{Kesahorm} Kesahorm, T.; Sintunavarat, W. On novel common fixed point results for enriched nonexpansive semigroups. Thai J. Math. 18 (2020), no. 3, 1549--1563.

\bibitem{Kwun} Kwun, Y. C.; Qadri, H. M. U. A.; Nazeer, W.; Haq, Absar U.; Kang, S. M. On generalized moduli of quasi-Banach space. {\em J. Funct. Spaces} {\bf 2018}, Art. ID 7324783, 10 pp.

\bibitem{Maia} Maia, M., Un'osservazione sulle contrazioni metriche. {\em Rend. Sem. Mat. Univ. Padova} {\bf 40} (1968), 139--143.

\bibitem{Mal} Maligranda, L. Type, cotype and convexity properties of quasi-Banach spaces. Banach and function spaces, 83--120, Yokohama Publ., Yokohama, 2004.

\bibitem{Mic} Miculescu, R.; Mihail, A. A generalization of Matkowski's fixed point theorem and Istr\u a\c tescu's fixed point theorem concerning convex contractions. {\it J. Fixed Point Theory Appl.} {\bf 19} (2017), no. 2, 1525--1533.

\bibitem{Mic17a} Miculescu, R.; Mihail, A. Caristi-Kirk type and Boyd\&Wong-Browder-Matkowski-Rus type fixed point results in $b$-metric spaces. {\it Filomat} {\bf 31} (2017), no. 14, 4331--4340.


\bibitem{Mitrea} Mitrea, D.; Mitrea, I.; Mitrea, M.; Monniaux, S. {\it Groupoid metrization theory. With applications to analysis on quasi-metric spaces and functional analysis}. Applied and Numerical Harmonic Analysis. Birkh\" auser/Springer, New York, 2013.

\bibitem{Mondal} Mondal, P.; Garai, H.; Dey, L. K. On some enriched contractions in Banach spaces and an application. {\em Filomat} {\bf 35} (2021), no. 15, 5017--5029.

\bibitem{Pacurar} P\u acurar, M. Asymptotic stability of equilibria for difference equations via fixed points of enriched Pre\v si\' c operators. {\em Demonstr. Math.} {\bf 56} (2023), no. 1, Paper No. 20220185, 8 pp.

\bibitem{Panicker} Panicker, R.; Shukla, R. Stability results for enriched contraction mappings in convex metric spaces. {\em Abstr. Appl. Anal.} {\bf 2022}, Art. ID 5695286, 7 pp.

\bibitem{Panwar} Panwar, A.; Lamba, P.; Rako\v cevi\' c, V.; Gopal, D. New fixed point results of some enriched contractions in CAT(0) spaces. {\em Miskolc Math. Notes} {\bf 24} (2023), no. 3, 1477--1493.


\bibitem{Phairat} Phairatchatniyom, P.; Kumam, P.; Berinde, V. A modified Ishikawa iteration scheme for $b$-enriched nonexpansive mapping to solve split variational inclusion problem and fixed point problem in Hilbert spaces. Math. Methods Appl. Sci. 46 (2023), no. 12, 13243--13261.

\bibitem{Pietsch} Pietsch, A. {\it History of Banach spaces and linear operators}. Birkh\" auser Boston, Inc., Boston, MA, 2007.

\bibitem{Rawat} Rawat, S.; Kukreti, S.; Dimri, R. C. Fixed point results for enriched ordered contractions in noncommutative Banach spaces. {\em J. Anal.} {\bf 30} (2022), no. 4, 1555--1566.

\bibitem{Rus01a} Rus, I.A. {\em Generalized Contractions and Applications}. Cluj-University Press,
Cluj-Napoca, 2001.

\bibitem{Rus03} Rus, I.A. Picard operators and applications. {\em Sci. Math. Jpn.} \textbf{58} (2003), no. 1, 191--219.

\bibitem{Rus14} Rus, I.A. Heuristic introduction to weakly Picard operator theory. {\em Creat. Math. Inform.} \textbf{23} (2014), no. 2, 243--252.

\bibitem{Salisu} Salisu, S.; Kumam, P.; Sriwongsa, S.; Gopal, D. Enriched asymptotically nonexpansive mappings with center zero. {\em Filomat} {\bf 38} (2024), no. 1, 343--356.


\bibitem{San} S\' anchez, F. C. An example regarding Kalton's paper 'Isomorphisms between spaces of vector-valued continuous functions'. {\em Proc. Edinb. Math. Soc. (2)} {\bf 64} (2021), no. 3, 615--619.

\bibitem{Shuk} Shukla, R.; Panicker, R. Generalized enriched nonexpansive mappings and their fixed point theorems. {\em Abstr. Appl. Anal.} {\bf 2023}, Art. ID 5572893, 10 pp.

\bibitem{Shuk-2} Shukla, R.; Panicker, R.  Some fixed point theorems for generalized enriched nonexpansive mappings in Banach spaces. {\em Rend. Circ. Mat. Palermo (2)} {\bf 72} (2023), no. 2, 1087--1101.

\bibitem{Shuk-3} Shukla, R.; Pant, R. Some fixed point results for enriched nonexpansive type mappings in Banach spaces. {\em Appl. Gen. Topol.} {\bf 23} (2022), no. 1, 31--43.

\bibitem{Suan} Suantai, S.; Chumpungam, D.; Sarnmeta, P. Existence of fixed points of weak enriched nonexpansive mappings in Banach spaces. {\em Carpathian J. Math.} {\bf 37} (2021), no. 2, 287--294.

\bibitem{Sul} Sullivan, T. J. Well-posed Bayesian inverse problems and heavy-tailed stable quasi-Banach space priors. {\em Inverse Probl. Imaging} {\bf 11} (2017), no. 5, 857--874.

\bibitem{Suz17} Suzuki, T. Basic inequality on a $b$-metric space and its applications. {\it J. Inequal. Appl.} {\bf 2017}, Paper No. 256, 11 pp.

\bibitem{Tor} Torres, R. H.; Xue, Q. Y.; Yan, J. Q. Compact bilinear commutators: the quasi-Banach space case. {\em J. Anal.} {\bf 26} (2018), no. 2, 227--234.

\bibitem{Tych} Tychonoff, A. Ein Fixpunktsatz. (German) {\it Math. Ann.} {\bf 111} (1935), 767--776.

\bibitem{Vulpe} Vulpe, I. M.; Ostrajkh, D.; Khojman, F. The topological structure of a quasi-metric space. (Russian) in 
{\it Investigations in functional analysis and differential equations}, Math. Sci., Interuniv. Work Collect., Kishinev 1981, 14--19 (1981).

\bibitem{Wei} Wei, W. Z. Uniformly PL-convex property for complex quasi-Banach space and q-mean square function of symmetric complex martingales. (Chinese) {\em Guangxi Sci.} {\bf 6} (1999), no. 3, 161--164.


\bibitem{Zada} Zada, A.; Mashal, A. Stability analysis of $n$th order nonlinear impulsive differential equations in quasi-Banach space. {\em Numer. Funct. Anal. Optim.} {\bf 41} (2020), no. 3, 294--321.


\end{thebibliography}
\end{document}